\documentclass[reqno,11pt]{amsart}
\usepackage[foot]{amsaddr}
\usepackage{bbold}
\usepackage{charter}
\usepackage[margin=1in]{geometry}
\usepackage[colorlinks=true,linktoc=all,linkcolor=blue,citecolor=blue]{hyperref}

\theoremstyle{plain}
\newtheorem{theorem}{Theorem}[section]
\newtheorem*{theorem*}{Theorem}
\newtheorem{proposition}{Proposition}[theorem]
\newtheorem{lemma}[theorem]{Lemma}
\newtheorem{corollary}[theorem]{Corollary}
\newtheorem*{claim}{Claim}
\theoremstyle{definition}
\newtheorem{defi}[theorem]{Definition}

\newtheorem*{definition*}{Definition}

\newtheorem{example}[theorem]{Example}

\numberwithin{equation}{section}
\everymath{\displaystyle}
\allowdisplaybreaks

\newcommand\C{\mathbb C}
\newcommand\D{\mathbb D}
\newcommand\R{\mathbb R}
\newcommand\N{\mathbb N}
\newcommand\Q{\mathbb Q}
\newcommand\calS{\mathcal S}
\newcommand\calL{\mathcal L}
\newcommand\calF{\mathcal F}
\newcommand\Lmu{L_\mu(T)}
\newcommand\Lmuz{L_\mu^0(T)}

\renewcommand\o{\operatorname{o}}
\newcommand\dist{\operatorname{d}}

\renewcommand\phi{\varphi}
\renewcommand\epsilon{\varepsilon}
\newcommand*{\bigchi}{\mbox{\Large$\chi$}}
\newcommand\ben{\begin{eqnarray}}
\newcommand\eeqn{\end{eqnarray}}
\newcommand\ds{\displaystyle}

\keywords{Composition operators; Hypercyclicity; Weighted Banach spaces}

\title[Hypercyclicity of $C_\varphi$ on $\Lmuz$]{Hypercyclicity of
	composition operators on\\ discrete weighted Banach spaces}

\author{Robert F.~Allen\textsuperscript{1}, Flavia Colonna\textsuperscript{2}, Rub\'en A.~Mart\'inez-Avenda\~no\textsuperscript{3}, and Matthew A.~Pons\textsuperscript{4}}
\address{\textsuperscript{1}Department of Mathematics and Statistics, University of Wisconsin-La Crosse, USA}
\address{\textsuperscript{2}Department of Mathematical Sciences, George Mason University, USA}
\address{\textsuperscript{3}Departamento Acad\'emico de Matem\'aticas, Instituto Tecnol\'ogico Aut\'onomo de M\'exico, Mexico}
\address{\textsuperscript{4}Department of Mathematics, North Central College, USA} 

\email{rallen@@uwlax.edu, fcolonna@gmu.edu, rubeno71@gmail.com, mapons@noctrl.edu}

\begin{document}

\begin{abstract}
	In this paper, we study the hypercyclic composition operators
        on weighted Banach spaces of functions defined on discrete
        metric spaces.  We show that the only such composition
        operators act on the ``little'' spaces. We characterize the
        bounded composition operators on the little spaces, as well as
        provide various necessary conditions for hypercyclicity.
\end{abstract}

\maketitle

\section{Introduction}
Let $X$ be a Banach space of functions on a domain $\Omega$. A
self-map $\varphi$ of $\Omega$ induces the composition operator
defined by
$$
C_\varphi f = f\circ\varphi
$$
for $f \in X$.  We denote the set of analytic functions on the open
unit disk $\D$ of the complex plane by $H(\D)$.  When $X$ is $H(\D)$,
the research of composition operators is extensive, beginning with the
work of Nordgren \cite{Nord}.  For an excellent treatise on the
subject, the reader is directed to \cite{CowMac}.

One such function space of importance is the so called weighted Banach
space.  For a bounded and continuous function $v:\D \to (0,\infty)$,
define the weighted Banach space $H_v^\infty$ by
$$
H_v^\infty := \left\{f \in H(\D) : \|f\|_v = \sup_{z \in \D}
  v(z)|f(z)| < \infty\right\}.
$$
Under the norm $\|\cdot\|_v$, the set $H_v^\infty$ is a Banach space
that arises naturally in the fields of complex analysis, Fourier
analysis, and partial differential equations, see
\cite{BBG-1,BBG-2,BBG-3,Lus}.

In recent years, the study of operators on so called discrete function
spaces has increased.  Most recently, much work has been conducted on
spaces where $\Omega$ is an infinite, rooted tree.  In this setting,
discrete analogs to some classical function spaces on $\D$ have been
developed and operators studied, including the Bloch space
\cite{AllenColonnaEasley:2014,ColonnaEasley:2010}, the Hardy space
\cite{MuthukumarPonnusamy:2016,MuthukumarPonnusamy-I:2016}, and the
weighted Banach spaces
\cite{AllenCraig:2015,AllenPons:2016,AllenPons:2018}.

The study of linear dynamics is an increasingly active area of
research.  For a treatise on the subject, see for example
\cite{GrErPeris}. Connected to the study of such dynamics is the
question of when a composition operator $C_\varphi$ is hypercyclic.
Hypercyclic composition operators $C_\varphi$ on the weighted Banach
space $H_v^0$ were investigated in \cite{MirWolf}, while hypercyclic
scalar multiples of composition operators on $H_v^0$ were studied in
\cite{LiangZhou}. In the environment of directed trees, the
hypercyclicity of shifts on weighted $L^p$ spaces has been studied in
\cite{MA}.

For a recent exposition on hypercyclic composition operators on Banach
spaces of analytic functions, see \cite{CoMA}.  The purpose of this
paper is to bring the study of hypercyclic composition operators to
the setting of discrete function spaces, specifically weighted Banach
spaces.  In addition to this, we generalize the domains of the
function spaces from an infinite, rooted tree to an unbounded, locally
finite metric space with a distinguished point.

\subsection{Organization of the paper.}
In Section~\ref{Section:Weighted Spaces}, we define the weighted
Banach space $L_\mu(T)$ and little weighted Banach space $L_\mu^0(T)$
on an unbounded, locally finite metric space $T$ with a distinguished
element, with respect to a weight (i.e. a positive function) $\mu$ on
$T$. We also collect useful facts about these spaces.

In Sections~\ref{Section:BoundedBigSpace} and
\ref{Section:BoundedLittleSpace}, we characterize the bounded
composition operators acting on the weighted and little weighted
Banach spaces, respectively.  Finally in
Section~\ref{Section:Hypercyclic}, we show that no composition
operator is hypercyclic on $\Lmu$, and provide many necessary
conditions for $C_\varphi$ to be hypercyclic on $\Lmuz$.  We conclude
the section with an example of a hypercyclic composition operator on
$\Lmuz$.

\section{Weighted spaces}\label{Section:Weighted Spaces}

We start by defining a locally finite metric space and the two
function spaces that we will study throughout this paper.

\begin{defi}
Let $(T,\dist)$ be a metric space with a distinguished element
$\o$. Define $|v|:=\dist(\o,v)$. We say that $T$ is \textbf{locally
  finite} if for each $M>0$, the set $\{ v \in T \, : \, |v| \leq M\}$ is finite.
\end{defi}

\begin{defi}
Let $(T,\dist)$ be an unbounded, locally finite metric space with a
distinguished element $\o$ and let $\mu$ be a positive function on
$T$. We define the \textbf{weighted Banach space} on $T$ as the set
$\Lmu$ of functions $f: T \to \C$ such that
$$
\| f \|_\mu := \sup_{v\in T} \mu(v) |f(v)| < \infty.
$$
The {\it little weighted Banach space} is the subset $\Lmuz$ of $\Lmu$
whose elements $f$ satisfy the condition
$$
\lim_{|v| \to \infty} \mu(v) |f(v)| = 0.
$$
\end{defi}

First, we observe that the point evaluation functionals on $\Lmu$ are bounded.

\begin{lemma}\label{Lemma:boundedevaluation} Let $(T,\dist)$ be an
  unbounded, locally finite metric space with a distinguished element
  $\o$ and let $\mu$ be a positive function on $T$.  For $v \in T$ and $f \in \Lmu$, we
  have $$|f(v)| \leq \frac{1}{\mu(v)} \| f \|_\mu.$$ Moreover, the
  point evaluation functional $e_v$ is bounded on $\Lmu$.
\end{lemma}

\begin{proof}
For $v \in T$, observe
$$
|f(v)| \leq \frac{1}{\mu(v)} \sup_{u\in T} \mu(u) |f(u)| =
\frac{1}{\mu(v)} \| f \|_\mu.
$$
Thus, for $f \in \Lmu$ with $\|f\|_\mu \leq 1$, it follows
that
$$
\|e_v\| = \sup_{\|f\|_\mu = 1} \|e_v(f)\| = \sup_{\|f\|_\mu = 1} |f(v)| \leq
\frac{1}{\mu(v)},
$$
and thus the evaluation functionals are bounded on $\Lmu$.
\end{proof}

As expected, the weighted Banach space and little weighted Banach
spaces are, in fact, Banach spaces under the norm $\|\cdot\|_\mu$.
The following proof is included for completeness.  A version of this
result when $T$ is an infinite tree and $\dist$ is the edge counting
metric can be found in \cite{AllenCraig:2015}.

\begin{theorem}\label{Propostion:Banach Spaces}
Let $(T,\dist)$ be an unbounded, locally finite metric space with a
distinguished element $\o$ and let $\mu$ be a positive function on
$T$. Then, under the norm $\|\cdot\|_\mu$, $\Lmu$ is a Banach space.
\end{theorem}

\begin{proof}
The set $\Lmu$ is clearly a normed linear space. To show completeness,
let $\{f_n \}$ be a Cauchy sequence in $\Lmu$. By Lemma
\ref{Lemma:boundedevaluation}, for each $v \in T$, the sequence $\{
f_n(v) \}$ is also Cauchy, and hence converges to a complex number
$f(v)$. Now, given $\epsilon >0$, there exists $N \in \N$ such that,
for $n, m \geq N$ we have
$$
\mu(v) | f_n(v) - f_m(v) | \leq \| f_n - f_m \|_\mu < \epsilon/2,
$$
for all $v \in T$. Letting $m$ go to infinity, we obtain
$$
\mu(v) | f_n(v) - f(v) | \leq \epsilon/2,
$$
for each $v \in T$. Taking the supremum over $v \in T$ we obtain
$\|f_n - f \|_\mu < \epsilon$ for all $n \geq N$. In particular, $f_N
- f$ is in $\Lmu$ and hence $f \in \Lmu$.
\end{proof}

Recall that a functional Banach space is a Banach space of complex-valued
functions on a set such that all point evaluations are bounded linear
functionals, no point evaluation functional is identically $0$ and the
point evaluation functionals separate points.

Theorem \ref{Propostion:Banach Spaces} shows that  $\Lmu$ is a Banach
space, Lemma \ref{Lemma:boundedevaluation} shows that all point
evaluations are bounded; clearly, no point evaluation is
identically zero and point evaluations separate points. Hence, $\Lmu$
is  a functional Banach space.

Next, we show $\Lmuz$ to be a closed and separable subspace of $\Lmu$.

\begin{proposition}
Let $(T,\dist)$ be an unbounded, locally finite metric space with a
distinguished element $\o$ and let $\mu$ be a positive function on
$T$. The set $\Lmuz$ is a closed subspace of $\Lmu$.
\end{proposition}

\begin{proof}
Let $\{ f_n \}$ be a sequence in $\Lmuz$ and assume $f_n \to f$ for
some $f \in \Lmu$. Let $\epsilon >0$ and choose $N \in \N$ such that
$\| f_N -f \|_\mu < \epsilon/2$. Let $M \in \N$ be chosen so that
$$
\mu (v) | f_N(v)| < \epsilon/2
$$
for all $v \in T$ with $|v|  \geq M$. Then, for all such $v$ we have
$$
\mu(v) |f(v)| \leq \mu(v) |f(v)-f_N(v)| + \mu(v) |f_N(v)| \leq \|
f_N-f\|_\mu + \mu(v) |f_N(v)| < \epsilon,
$$
which shows $f \in \Lmuz$.
\end{proof}

The space $\Lmuz$ is separable, as the next theorem shows.


\begin{theorem}
Let $(T,d)$ be an unbounded, locally finite metric space with a
distinguished element $\o$ and let $\mu$ be a positive function on
$T$. The set $$ \{ f: T \to \C \, : \, f \text{ has finite support}\,
\}$$ is dense in $\Lmuz$ and hence $\Lmuz$ is separable.
\end{theorem}

\begin{proof}
Let $f \in \Lmuz$. Given $\epsilon >0$, there exists $N \in \N$ such
that $$\mu(v) |f(v)| < \frac{\epsilon}{2},$$ for all $|v| \geq N$.
	
For $n \in \N$, let $U_n:=\{ v \in T \, : \, |v|< n\}$, which
is a finite set since $T$ is locally finite. Let $\bigchi_{U_n}$ be
the characteristic function of $U_n$ and define $g_n :=
f\bigchi_{U_n}$. Clearly, $g_n$ has finite support. Then, for each $n
\in \N$ we have
$$
\| f - g_n \|_\mu 
= \sup_{v \in T} \mu(v) |(f-g_n)(v)| 
= \sup_{v \in T} \mu(v) |f(v)| \,  |1-\bigchi_{U_n}(v)|  
= \sup_{|v| \geq n} \mu(v) |f(v)|,
$$
since  $1-\bigchi_{U_n}(v)=0$ whenever $|v| < n$.
	
Now, let $n \geq N$. Since if $|v|\geq N$ then $\mu(v) |f(v)| <
\frac{\epsilon}{2}$, we have
$$\| f - g_n \|_\mu = \sup_{|v| \geq n} \mu(v) |f(v)| <\epsilon,
$$
which proves that $\lim_{n \to \infty} g_n = f$, as desired.

Clearly the set $\{ f: T \to \C \, : \, f \text{ has finite support}
\}$ can be approximated by the set $\{ f: T \to \Q+i\Q \, : \, f
\text{ has finite support} \}$, which is countable since $T$ is countable. Hence $\Lmuz$ is
separable.
\end{proof}

On the other hand, $\Lmu$ is never separable: indeed, define
\[
\calF := \{ f \in \Lmu \, : \, f(v) \in \{0, 1/\mu(v) \} \text{ for
  each } v \in T \}.
\]
Clearly, if $f, g \in \calF$ and $f \neq g$, then $\| f - g\| =
1$. Since $\calF$ is uncountable, it follows (for example, see
\cite[Proposition 1.12.1]{meg}) that $\Lmu$ is not separable.

\section{Boundedness of Composition Operators on the big
  space}\label{Section:BoundedBigSpace}

In this section, we study when composition operators are bounded on
$\Lmu$ and we obtain a complete characterization of such operators.

\begin{theorem}\label{Theorem:Bounded Lmu}
Let $(T,\dist)$ be an unbounded, locally finite metric space with a
distinguished element $\o$, let $\mu$ be a positive function on $T$
and let $\phi$ be a self-map of $T$. The operator $C_\phi$ is bounded
on $\Lmu$ if and only if
$$
\sup_{v \in T}\frac{\mu(v)}{\mu(\phi(v))} < \infty.
$$
Furthermore, $\ds \| C_\phi \|=\sup_{v \in T}\frac{\mu(v)}{\mu(\phi(v))}$.
\end{theorem}

\begin{proof}
Let $\ds C:=\sup_{v \in T}\frac{\mu(v)}{\mu(\phi(v))}$.
	
First, assume that $C < \infty$. Let $f \in \Lmu$. Then
\begin{eqnarray*}
\| C_\phi f \|_\mu  &=& \sup_{v \in T} \mu(v) | f(\phi(v))| \\
&=&  \sup_{v \in T} \frac{\mu(v)}{ \mu(\phi(v))}  \mu(\phi(v)) | f(\phi(v))| \\
&\leq& C  \sup_{v \in T} \mu(\phi(v)) | f(\phi(v))| \\
&\leq& C  \sup_{w \in T} \mu(w) | f(w)| \\
&=& C \| f \|_\mu,
\end{eqnarray*}
which shows that $C_\phi$ is bounded and $\| C_\phi \| \leq C$.
	
Now, for $w \in T$ define $g_w(v) :=
\frac{1}{\mu(v)}\bigchi_{\varphi(w)}(v)$. Clearly $\|g_w \|_\mu
=1$. Then
$$\| C_\phi \|  = \sup_{\| f \|_\mu=1} \| C_\phi f \|_\mu \geq  \|
C_\phi g_w \|_\mu = \sup_{v \in T} \frac{\mu(v)}{\mu(\phi(v))}
\chi_{\phi(w)} (\phi(v)) \geq \frac{\mu(w)}{\mu(\phi(w))}.$$
Hence, if $C_\phi$ is bounded, then
$$
\sup_{w\in T} \frac{\mu(w)}{\mu(\phi(w))} < \infty
$$ 
and
$$
\sup_{w \in T}\frac{\mu(w)}{\mu(\phi(w))} \leq \| C_\phi \|,
$$ which finishes the proof.
\end{proof}

In the remainder of the paper, we shall restrict our attention to
composition operators that are bounded on the space $\Lmu$. For easier
reference, we adopt the following definition.

\begin{defi}
Let $(T,\dist)$ be an unbounded, locally finite metric space with a
distinguished element $\o$,  $\mu$ a positive function on $T$, and
$\phi$ be a self-map of $T$. We say $\varphi$ is an \textbf{admissible
  symbol} for a composition operator (or simply \textbf{admissible})
if  the composition operator $C_\varphi$ is bounded on $\Lmu$.  
\end{defi}

Note by Theorem ~\ref{Theorem:Bounded Lmu}, $\varphi$ is admissible if and
only if there exists a positive constant $C$ such that $$\mu(w) \leq C
\mu(\phi(w))$$ for all $w \in T$.

The following proposition shows that, with an additional hypothesis on
$\phi$, admissibility is characterized by boundedness of the weight.

\begin{proposition}\label{Proposition:mu bounded}
Let $(T,\dist)$ be an unbounded, locally finite metric space with a
distinguished element $\o$, let $\mu$ be a positive function on $T$
and let $\phi$ be a self-map of $T$. If $\phi$ has finite range, then
$\varphi$ is admissible if and only if $\mu$ is bounded.
\end{proposition}

\begin{proof} 
Assume $\varphi$ is admissible. Then, there exists $C>0$ such that
$\mu(v) \leq C \mu(\phi(v))$ for all $v\in T$. Since $\phi$ has finite
range, the set $\{\mu(\phi(v)) \, : \, v \in T \}$ is finite and hence
bounded. Hence $\mu$ is bounded as well.
	
Conversely, suppose $\mu$ is bounded with upper bound $M>0$. Since
$\varphi$ has finite range, $\delta:=\min_{v\in T}\mu(\varphi(v))>0$.
Thus, by Theorem~\ref{Theorem:Bounded Lmu},
$$
\|C_\varphi\|=\sup_{v\in T}\frac{\mu(v)}{\mu(\varphi(v))}\le
\frac{M}{\delta},
$$
proving that $\varphi$ is admissible.
\end{proof}

\section{Boundedness of Composition Operators on the little
  space} \label{Section:BoundedLittleSpace}

The first result in this section, which highlights the connection to
Proposition \ref{Proposition:mu bounded} for the big space case, gives
a necessary and sufficient condition for $C_\varphi$ to be bounded on
$\Lmuz$ whenever $\varphi$ has finite range.

\begin{proposition}\label{prop:finiterange}
Let $(T,\dist)$ be an unbounded, locally finite metric space with a
distinguished element $\o$, let $\mu$ be a positive function on $T$
and let $\phi$ be a self-map of $T$. If $\phi$ has finite range, then
$C_\phi$ is bounded on $\Lmuz$ if and only if $\ds \lim_{|v|\to \infty
}\mu(v)=0$.
\end{proposition}

\begin{proof}
Let $S=\phi(T)$ be a finite set. Suppose first that $C_\varphi$ is
bounded on $\Lmuz$, and for purposes of contradiction assume
$\lim_{|v|\to \infty} \mu(v)$ does not equal zero. Then there exist
$c>0$ and a sequence $\{ v_n \}$ in
$T$ such that $|v_n| \to \infty$ and $\mu(v_n) \geq c$. Since $\chi_S$
is in $\Lmuz$, then $C_\varphi(\bigchi_S) = \bigchi_S \circ \phi$ is in $\Lmuz$. But
$$
\mu(v_n) | \chi_S (\phi(v_n)) | = \mu(v_n) \not\to 0,
$$
which is a contradiction. Hence $\lim_{|v|\to \infty} \mu(v) =0$.

		
Conversely, suppose $\mu(v) \to 0$ as $|v| \to \infty$.  To show that $C_\varphi$ is bounded on $\Lmuz$ it suffices to
show that $f \circ \phi$ is in $\Lmuz$ for each $f \in \Lmuz$ by the
Closed Graph Theorem (see \cite[Exercise 1.1.1]{CowMac}). Fixing such
an $f$, let $M:=\max_{v \in T} |f(\phi(v))|$. Then
$$
0 \leq \lim_{|v| \to \infty} \mu(v) |f(\phi(v))| \leq \lim_{|v| \to
  \infty} M \mu(v) =0,
$$
and thus $f \circ \phi \in \Lmuz$, as desired.
\end{proof}

Observe that, unlike the case for $\Lmu$ (see
Theorem~\ref{Theorem:Bounded Lmu}) admissibility itself cannot be
enough to guarantee boundedness of $C_\phi$ on $\Lmuz$. For example,
Proposition~\ref{prop:finiterange}  guarantees that if $\phi$ has
finite range and $\mu$ is does not go to zero, then $C_\phi$ does not
map $\Lmuz$ into $\Lmuz$ even in the event that $\phi$ is
admissible. Nevertheless, admissibility of $\phi$ is a necessary
condition for the boundedness of the composition operator $C_\phi$ on
$\Lmuz$.

\begin{proposition}\label{Proposition:necessary} 
Let $(T,\dist)$ be an unbounded, locally finite metric space with a
distinguished element $\o$, let $\mu$ be a positive function on $T$
and let $\phi$ be a self-map of $T$. If $C_\varphi$ is bounded on
$\Lmuz$, then $\varphi$ is admissible.
\end{proposition}

\noindent This follows from the proof of Theorem~\ref{Theorem:Bounded
  Lmu} since the test function $g_w$ used is an element of $\Lmuz$.


By the previous result, to show that $C_\varphi$ is bounded on $\Lmuz$
it is permissible to assume $C_\varphi$ is bounded on $\Lmu$.  This
additional hypothesis allows us to replace the assumption of finite
range with admissibility.

\begin{proposition}\label{pro_mu_to_zero}
Let $(T,\dist)$ be an unbounded, locally finite metric space with a
distinguished element $\o$,  $\mu$ be a positive function on $T$ and
$\phi$ admissible. If $\lim_{|v|\to \infty} \mu(v) =0$, then $C_\phi$
is bounded on $\Lmuz$.
\end{proposition}

\begin{proof}
Let $C >0$ be a constant such that $\mu(w) \leq C \mu(\phi(w))$ for
all $w \in T$. Fix $\epsilon >0$ and suppose $f \in \Lmuz$. We can
then choose $N_1 \in \N$ such that
\begin{equation}\label{eq1}
\mu(v) |f(v)| < \frac{\epsilon}{C}, \qquad \hbox{ for all } |v| \geq N_1. 
\end{equation}
Let $S:=\{ v \in T \, : \, |v| \leq N_1 \}$, which is a finite set
since $T$ is locally finite, and let $\ds m:=\max_{v \in S} |f(v)|$.
Also, since $\mu(v) \to 0$ as $|v| \to \infty$, we can choose $N_2 \in
\N$ such that
\begin{equation}\label{eq2}
\mu(v) < \frac{\epsilon}{m+1}, \qquad \hbox{ for all } |v|\geq N_2.
\end{equation}
	
Now, if $|\phi(v)| \geq N_1$, then, by \eqref{eq1}, we have $$\mu(v)
|f(\phi(v))| \leq C \mu(\phi(v)) |f(\phi(v))| < C \frac{\epsilon}{C} =
\epsilon.$$ Also, if $|v|\geq N_2$ and $|\phi(v)| \leq N_1$, then, by
\eqref{eq2} and the definition of $m$, we have $$\mu(v) | f(\phi(v))|
\leq \mu(v) m \leq \frac{\epsilon}{m+1} m < \epsilon.$$ Therefore, if
$|v|\geq N_2$, we have $\mu(v) | f(\phi(v))| < \epsilon$, which shows
that $f\circ \phi \in \Lmuz$. Since $C_\phi$ is bounded on $\Lmu$,
this implies that $C_\phi$ is bounded on $\Lmuz$.
\end{proof}

The previous result gave a condition on $\mu$ that yields the
boundedness of the composition operator. The following proposition
gives instead a condition on the symbol $\phi$.

\begin{proposition}\label{pro_phi_minus_one_finite}
Let $(T,\dist)$ be an unbounded, locally finite metric space with a
distinguished element $\o$, $\mu$ be a positive function on $T$ and
$\phi$ admissible. If $\phi^{-1}(v)$ is a finite set for every $v \in
T$, then $C_\phi$ is bounded on $\Lmuz$.
\end{proposition}

\begin{proof}
We claim first that $\ds \lim_{|v|\to \infty} |\phi(v)|=\infty$. If
not, there would exist $M_0 >0$ and a sequence $\{ v_n \}$ in $T$ such
that $|v_n|\geq n$ and $|\phi(v_n)|< M_0$ for each $n\in\N$. But this
implies that there exist $w \in T$ with $|w|<M_0$ and an infinite
subset $V$ of $T$ such that $\phi(v)=w$ for each $v \in V$, which
contradicts the hypothesis.
	
Let $f \in \Lmuz$. Since $\phi$ is admissible, there exists $C>0$ such
that $\mu(v) \leq C \mu(\phi(z))$ for all $v \in T$. Then,
$$
\lim_{|v|\to \infty} \mu(v) |(f\circ \phi) (v)| \leq  \lim_{|v|\to
  \infty} C \mu(\phi(v)) |f(\phi(v))|  = C \lim_{|\phi(v)|\to \infty}
\mu(\phi(v)) |f(\phi(v))| = 0.
$$
Thus, $C_\phi$ maps $\Lmuz$ into itself. The boundedness of the
operator on  $\Lmuz$ follows at once from the admissibility of $\phi$.
\end{proof}

Observe that the previous proposition includes the case where the
function $\phi$ is injective. For example, $T$ might be an infinite
and locally finite graph with the edge-counting metric $d$, $\mu$ any
positive function on $T$ and $\phi$  any admissible automorphism of
the graph.

The following proposition allows $\phi^{-1}(v)$ to be infinite for
finitely many $v \in T$, but $\mu$ must tend to zero on the inverse
image of such vertices.

\begin{proposition}\label{pro_phi_minus_one_infinite}
Let $(T,\dist)$ be an unbounded, locally finite metric space with a
distinguished element $\o$, $\mu$ be a positive function on $T$ and
$\phi$ an admissible function. Consider the set $\calS:= \{ w \in T \,
: \, \phi^{-1}(w) \text{ is infinite} \}$ and assume $\calS$ is
  nonempty. If $\calS$ is finite and
$$
\lim_{\substack{|v|\to \infty \\ v \in \phi^{-1}(\calS)}} \mu(v) = 0,
$$
then $C_\phi$ is bounded on $\Lmuz$.
\end{proposition}

\begin{proof}
Let $f \in \Lmuz$. Since $\phi$ is admissible, we only need to show
that $f \circ \phi$ is in $\Lmuz$. Let $m:=\max\{ |f(w)| \, : \, w \in
\calS \}$ and let $C>0$ such that $\mu(v) \leq C \mu(\phi(v))$ for all
$v \in T$.
	
Let $\epsilon>0$. Since
$$
\lim_{\substack{|v|\to \infty \\ v \in \phi^{-1}(\calS)}} \mu(v) = 0,
$$
there exists $N_1 \in \N$ such that
$$
\mu(v)  < \frac{\epsilon}{m+1}, \quad \text{ for all } v \in
\phi^{-1}(\calS) \text{ with } |v| \geq N_1.
$$
We consider the two cases when the complement of $\phi^{-1}(\calS)$ is finite or infinite.

Assume first $T\setminus \phi^{-1}(\calS)$ is finite. Let $N_2 \in \N$
  such that if $|v|\geq N_2$ then $v \notin T \setminus
  \phi^{-1}(\calS)$.
		
Define $N=\max\{N_1, N_2\}$. Then, if $|v|\geq N$, it follows that $v
\in \phi^{-1}(\calS)$. Hence,
$$
\mu(v) |f(\phi(v))| \leq \mu(v) m < \epsilon.
$$ Hence $f\circ \phi \in \Lmuz$.

Next, assume $T\setminus \phi^{-1}(\calS)$ is infinite. Since $f \in
  \Lmuz$, there exists $N_3 \in \N$ such that
$$
\mu(v) |f(v)| < \frac{\epsilon}{C}, \quad \text{ for all } v \in T
\text{ with } |v| \geq N_3.
$$
		
{\bf Claim:} There exists $N_4 \in \N$ such that for all $|v|\geq N_4$
with $v \notin \phi^{-1}(\calS)$ we have $|\phi(v)|\geq N_3$.
		
If not, there exists a sequence $\{ v_n \}$ of distinct points such
that $|v_n|\geq n$, such that $v_n \notin \phi^{-1}(\calS)$ and
$|\phi(v_n)| < N_3$ for all $n \in \N$. Passing to a subsequence if
necessary, it follows that there exists $w \in T$ with $|w| < N_3$ and
$\phi(v_n)=w$ for all $n \in \N$. But this implies that $w \in \calS$
and hence that $v_n \in \phi^{-1}(\calS)$ for all $n \in \N$,
contradicting the choice of the sequence and proving the claim.
		
Define $N=\max\{N_1, N_4\}$. Let $v \in T$ with $|v|\geq N$. If $v \in
\phi^{-1}(\calS)$, then
$$
\mu(v) |f(\phi(v))| \leq \mu(v) m < \epsilon.
$$
If $v \notin \phi^{-1}(\calS)$, then
$$
\mu(v) |f(\phi(v))| \leq C \mu(\phi(v)) |f(\phi(v))| < \epsilon,
$$
since $|\phi(v)|\geq N_3$. In either case,
$$
\mu(v) |f(\phi(v))| <\epsilon,
$$
and hence $f \circ \phi \in \Lmuz$.
\end{proof}

Proposition~\ref{scond_for_bddC_phi_onL0mu_v2} below generalizes the conditions of Propositions \ref{pro_mu_to_zero}, \ref{pro_phi_minus_one_finite} and
\ref{pro_phi_minus_one_infinite}. Indeed, in its statement we assume a condition that uses an arbitrary positive increasing function $g$ on the nonnegative integers. If $g$ is bounded, the result reduces to  Proposition~\ref{pro_mu_to_zero}. Moreover, in the proof of Proposition~\ref{pro_phi_minus_one_finite} we showed that $|\phi(v)|\to\infty$ as $|v|\to\infty$. Thus, letting $g$ be the identity, the result follows from  Proposition~\ref{scond_for_bddC_phi_onL0mu_v2}. The same argument applies to  Proposition~\ref{pro_phi_minus_one_infinite}, where the assumption combines elements of the previous two propositions. 

\begin{proposition}\label{scond_for_bddC_phi_onL0mu_v2}
Let $(T,\dist)$ be an unbounded, locally finite metric space with
distinguished element $o$, $g:\N_0\to \R_+$ an 
 increasing
function, $\mu$ a positive function on $T$ and $\phi$ an admissible
function. If
\begin{equation} 
\lim_{|v|\to\infty} \min\{\big(g(|\phi(v)|)\big)^{-1}\hskip-2pt, \mu(v) \} = 0,\label{bddC_phi_v2}
\end{equation}
then $C_\phi$ is bounded on $\Lmuz$.
\end{proposition}

\begin{proof} Let $C>0$ such that $\mu(v) \leq C \mu(\phi(v))$ for all $v \in
T$. Let $f \in \Lmuz$ and $\epsilon >0$. There exists $N_1 \in \N$
such that
\begin{equation}\label{eq:def_f}
\mu (v) |f(v)| < \frac{\epsilon}{C}, \quad \text{ for all $v$ with } |v| \geq N_1.
\end{equation}
Let $m:=\max\{ |f(v)| \, : \,  |v| \leq N_1\}$. By hypothesis, there
exists $N_2 \in \N$ such that
$$
\min\{ \big(g(|\phi(v)|)\big)^{-1}, \mu(v) \}  < \min\left\{
  \frac{\epsilon}{m+1}, \frac{1}{g(N_1)} \right\} , \quad \text{ for
  all $v$ with } |v| \geq N_2.
$$	
Now, let $v$ be such that $|v| \geq \max\{ N_1, N_2\}$.
	
\begin{itemize}
\item Assume $\min\{ \big(g(|\phi(v)|)\big)^{-1}, \mu(v) \}  = \mu(v)$. 
\smallskip
		
If $|\phi(v)| \geq N_1$, then, by the admissibility of $\phi$ and equation \eqref{eq:def_f}, we have
$$
\mu(v) |f(\phi(v))| \leq C \mu(\phi(v)) |f(\phi(v))| < \epsilon.
$$
If $|\phi(v)| \leq N_1$, then
$$
\mu(v) |f(\phi(v))| \leq \mu(v) m  < \epsilon,
$$
by the definition of $m$ and  since $\min\{
\big(g(|\phi(v)|)\big)^{-1}, \mu(v) \} = \mu(v)  < \dfrac{\epsilon}{m+1}$.

\item Assume $\min\{\big(g(|\phi(v)|)\big)^{-1}, \mu(v) \}  =
\big(g(|\phi(v)|)\big)^{-1}$. Then, $\frac1{g(|\phi(v)|)} <
\frac{1}{g(N_1)}$ and hence $|\phi(v)| > N_1$. Therefore,
$$
\mu(v) |f(\phi(v))| \leq C \mu(\phi(v)) |f(\phi(v))| < \epsilon,
$$
as before.
\end{itemize}
In either case, $\mu(v) |f(\phi(v))| < \epsilon$ and hence $f \circ \phi \in \Lmuz$.
\end{proof}

In the following example, we cannot apply Propositions
\ref{pro_mu_to_zero}, \ref{pro_phi_minus_one_finite}, or
\ref{pro_phi_minus_one_infinite}, but
Proposition~\ref{scond_for_bddC_phi_onL0mu_v2} yields the boundedness
of the operator.

\begin{example}
Consider the set $T=\N_0\times \N_0$ with the metric $\dist$ given by
$$\dist((m,n),(m',n'))=|m-m'|+|n-n'|$$ and take as a distinguished element the
pair $(0,0)$. Define $\mu$ on $T$ by $\mu(m,n)=3^m 2^{-n}$. The function
$\phi$ given by $\phi(m,n)=(m,0)$ is easily seen to be admissible. Since
$$
\lim_{|m|+|n|\to \infty} \min \{ (|m|+1)^{-1}, 3^m2^{-n}\} =0,
$$
condition (\ref{bddC_phi_v2}) holds for $g(x)=x+1$, $x \in \N_0$.  
Hence, $C_\phi$ is a bounded operator on $\Lmuz$.
\end{example}

Note that Proposition~\ref{scond_for_bddC_phi_onL0mu_v2} continues to
hold if we assume $g$ to be eventually increasing unboundedly.  We now
show that a converse of Proposition~\ref{scond_for_bddC_phi_onL0mu_v2}
holds under a restriction on the growth of the weight $\mu$. In
particular, we obtain a full characterization of boundedness for
$C_\phi$ acting on $L^0_\mu(T)$ if the weight $\mu$ is bounded.

\begin{theorem}\label{th:conv_for_bddC_phi_onL0mu_v2}	
Let $(T,\dist)$ be an unbounded, locally finite metric space with a
distinguished element $\o$, $g:\N_0\to \R_+$ an arbitrary 
 increasing
function, $\mu$ be a positive function on $T$ such that
$\mu(v)=o(g(|v|))$ as $|v|\to\infty$, and $\phi$ a self-map of
$T$. Then $C_\phi$ is bounded on $\Lmuz$ if and only if $\phi$ is
admissible and condition (\ref{bddC_phi_v2}) holds.
\end{theorem}

\begin{proof} If $g$ is bounded, the result is trivial, since the assumption $\mu=o(g)$ yields $\mu(v)\to 0$ as $|v|\to \infty$, and the conclusion follows at once applying Propositions~\ref{pro_mu_to_zero} and {Proposition:necessary}. 

Thus, let us  assume $g$ is unbounded. By Propositions~\ref{Proposition:necessary} and
  \ref{scond_for_bddC_phi_onL0mu_v2}, it suffices to show that if
  condition (\ref{bddC_phi_v2}) fails, then $C_\phi$ is not bounded on
  $\Lmuz$. Assume (\ref{bddC_phi_v2}) fails. Then there exist a
  sequence of vertices $\{v_n\}$ with $|v_n|\to \infty$ as
  $n\to\infty$, a positive constant $\delta$ and a positive integer
  $N$ such that $\min\{\big(g(|\phi(v_n)|)\big)^{-1},\mu(v_n)\}\ge
  \delta$ whenever $|v_n|\ge N$. Define
$$
f(v)=\begin{cases} \frac1{g(|v|)}, &\quad\text{ if }v\ne \o,\\
  0, &\quad\text{ if }v= \o.\end{cases}
$$
On the one hand, since $\mu(v)=o(g(|v|))$ as $|v|\to\infty$, it
follows that $f\in L^0_\mu(T)$. On the other hand, for all $n\in \N$
such that $|v_n|\ge N$,
$$
\mu(v_n)|f(\phi(v_n))|=\mu(v_n)\big(g(|\phi(v_n)|)\big)^{-1}\ge
\delta^2.
$$
Therefore, $C_\phi f\notin L^0_\mu(T)$.

\end{proof}

Are there conditions on the symbol $\phi$ and on the weight $\mu$ that
are both sufficient and necessary for boundedness of $C_\phi$ on
$\Lmuz$? We have not been able to find an answer to this problem so we
leave it open for future research.

\section{Hypercyclicity}\label{Section:Hypercyclic}

Let $X$ be a Banach space. Recall that an operator $S:X \to X$ is
called {\em hypercyclic} if there exists $x \in X$ such that $\{ x,
Sx, S^2 x, S^3x, \dots\}$ is dense in $X$. Such a vector $x$ is also
called hypercyclic. Clearly, the existence of a hypercyclic operator
on a Banach space $X$ implies that $X$ is separable. Observe that the
set of hypercyclic vectors is a dense set in $X$, since every vector
in the orbit $\{ x, Sx, S^2 x, S^3x, \dots\}$ is itself hypercyclic.

There has been much research on hypercyclicity and the reader is
referred to the books \cite{BaMa,GrErPeris} for the necessary
background on the study of hypercyclicity. In what follows, we
concentrate on the problem of determining hypercyclicity of
composition operators on $\Lmuz$. This study somewhat follows the
classical studies of hypercyclicity of composition operators on spaces
of analytic functions (see \cite{CoMA} for a review of hypercyclicity
of composition operators on several Banach spaces of analytic
functions on the disk).

As it is usually the case for composition operators, if $\phi: T \to
T$ has a fixed point, then $C_\phi$ cannot be hypercyclic on $\Lmuz$,
as we now show. A more general result can be obtained: a self-map of
$T$ with periodic points cannot induce a hypercyclic composition
operator on a functional Banach space.

\begin{theorem}
Let $X$ be a functional Banach space of complex-valued functions
defined on a set $T$. Let $\phi$ be a self-map of $T$ and assume that
$C_\phi$ is bounded on $X$. If $C_\phi$ is hypercyclic on $X$, then
$\phi$ does not have periodic points.
\end{theorem}

\begin{proof}
Suppose that there exists $w \in T$ and $p \in \N$ such that
$\phi^p(w)=w$. Let $f \in X$ be a hypercyclic vector. Then the
set
$$\{  f, f \circ \phi, f \circ \phi^2, f \circ \phi^3, \dots  \}
$$
is dense in $X$. Since the evaluation functional $e_w$ is continuous
and surjective, it follows that
$$
\{  f(w), f (\phi(w)), f (\phi^2(w)),  f(\phi^3(w)), \dots  \}
$$
is dense in $\C$. Since $\phi^p(w)=w$ this set equals
$$
\{  f(w), f(\phi(w)), f(\phi^2(w)), \dots,  f(\phi^{p-1}(w)) \}.
$$
But this set cannot be dense in $\C$.
\end{proof}

It follows that if the self-map $\phi$ of  $T$ has a fixed point (or a
periodic point) and $C_\phi$ is bounded, then $C_\phi$ is not
hypercyclic on the functional Banach space $\Lmuz$.

The following definition was introduced by Bernal-Gonz\'alez and
Montes-Rodr\'iguez in a different, but similar, context.

\begin{defi}
Let $(X,d)$ be a metric space and let $\phi$ be a self-map of $X$. We
say that $\phi$ is a \textbf{run-away} function if for every finite set
$I$, there exists $N \in \N$ such that $\phi^n(I) \cap I = \emptyset$
for every $n \geq N$.
\end{defi}

\begin{proposition}
Let $(X,d)$ be a metric space and let $\phi$ be a self-map of $X$. If
$\phi$ does not have periodic points, then $\phi$ is a run-away
function.
\end{proposition}

\begin{proof}
Assume that $\phi$ is not a run-away function. Therefore, there exists
a finite set $I$ and an increasing sequence $\{n_k\}$ of natural
numbers such that
$$
\phi^{n_k}(I) \cap I \neq \emptyset
$$
for all $k$. Since $I$ is finite, there exists $w \in I$ such that
$$
w \in \phi^{n_k}(I)
$$
for infinitely many natural numbers $k$. Now, since $I$ is finite, it
also follows that there exists $v \in I$ such that
$$
w = \phi^{n_k}(v)
$$
for infinitely many natural numbers $k$. Assume then that
$w=\phi^s(v)$ and $w=\phi^t(v)$, for integers $s$ and $t$ with
$s>t$. But then
$$
\phi^{s-t}(w)= \phi^{s-t}(\phi^t(v))=\phi^s(v)=w,
$$
which implies that $w$ is a periodic point, finishing the proof.
\end{proof}

Let $X$ be a functional Banach space over a set $T$ and assume
$C_\phi$ is a composition operator on $X$. Then the adjoint of the
operator $C_\phi$ maps point-evaluation functionals to
point-evaluation functionals. Specifically, for each $h\in X$ and each
$v\in T$, we have
$$
(C_\phi^* e_v)(h)= e_v (C_\phi h)=e_v (h\circ \phi)= h(\phi(v)) = e_{\phi(v)} (h),
$$
and hence $C_\phi^* e_v = e_{\phi(v)}$ for each $v\in T$.

\begin{theorem}
Let $X$ be a functional Banach space of complex-valued functions
defined on a set $T$. Let $\phi$ be a self-map of $T$ and assume that
$C_\phi$ is bounded on $X$. If $\phi$ is not injective, then $C_\phi$
is not hypercyclic on $X$.
\end{theorem}

\begin{proof}
 Assume $\phi(v)=\phi(w)$ for two distinct
vertices $v$ and $w$.  Let $e_v$ and $e_w$ be the evaluation
functionals for $v$ and $w$ respectively. Then $C_\phi^*(e_v-e_w)=0$
and hence $C_\phi^*$ has nontrivial kernel. This implies that the
range of $C_\phi$ cannot be dense and hence $C_\phi$ cannot be
hypercyclic.
\end{proof}

We should point out that if $\mu(v) \leq \mu(\phi(v))$ for all $v \in
T$ then $\| C_\phi \| \leq 1$, which implies that $C_\phi$ cannot be hypercyclic. 

The following theorem gives another condition which prevents hypercyclicity.

\begin{theorem}
Let $(T,\dist)$ be an unbounded, locally finite metric space with a
distinguished element $\o$, let $\mu$ be a positive function on $T$
and let $\phi$ be a self-map of $T$. Assume that $C_\phi$ is bounded
on $\Lmuz$. If $\phi$ is surjective and $\mu(\phi(v)) \leq \mu(v)$ for
every $v \in T$, then $C_\phi$ is not hypercyclic.
\end{theorem}
\begin{proof}
Observe that, for $f \in \Lmuz$ and for each $v \in T$ we have
$$
\mu(v) | (f \circ \phi) (v)| \geq \mu(\phi(v)) |f(\phi(v))|
$$
and hence
$$
\| C_\phi f \|_\mu \geq  \mu(\phi(v)) |f(\phi(v))|.
$$
Since $\phi$ is surjective, this implies that $\| C_\phi f \|_\mu
\geq \| f \|_\mu$ and hence that $\| C^n_\phi f \|_\mu \geq \| f
\|_\mu$ for each $n \in \N$. It follows that $C_\phi$ cannot be hypercyclic.
\end{proof}

Let $n \in \N$ and let $\phi$ be a self-map of $T$. We denote by $T^n$
the image of $\phi^n$; that is,
\[
T^n:=\{ w \in T \, : \, w=\phi^n(v) \text{ for some } v \in T \}.
\]
We also denote by $T^\infty$ the set
\[
T^\infty:=\{ w \in T \, : \, \text{ for each } n\in \N \text{ there
	exists } v  \in T \text{ with } \phi^n(v)=w \}.
\]
Clearly, $\ds T^\infty=\bigcap_{n=1}^\infty T^n$.

For the proof of our next theorem, we will need the so-called  {\em Hypercyclicity
  Criterion}, which was obtained independently by Kitai in 1982 and
by Gethner and Shapiro, in 1987. A modern proof can be found in \cite[Theorem 3.12]{GrErPeris}.

\begin{theorem*}[Hypercyclicity Criterion]
Let $\calL$ be a separable Banach space and $B$ a bounded operator on
$\calL$. Assume that there exists a dense subset $X$ of $\calL$, a
sequence $\{ n_k \}$ of increasing positive integers, and a sequence
of functions $\{ S_{n_k} \}$, with $S_{n_k} : X \to \calL$ such that, for
every $x \in X$,
\begin{enumerate}
\item $B^{n_k} x \to 0$,
\item $S_{n_k} x \to 0$, and
\item $B^{n_k} S_{n_k} x \to x$.
\end{enumerate}
Then, $B$ is hypercyclic.
\end{theorem*}

While the following three theorems can be obtained as corollaries to
\cite[Theorem 3]{GrEr}, we include their proofs, which were,
partially, obtained independently, and are specific to composition
operators on $\Lmuz$.

\begin{theorem}\label{t5.8}
	Let $(T,\dist)$ be an unbounded, locally finite metric space with a
	distinguished element $\o$, let $\mu$ be a positive function on $T$
	and let $\phi$ be a self-map of $T$. Assume that $C_\phi$ is bounded
	on $\Lmuz$ and $\phi$ is injective. Then, $C_\phi$ is hypercyclic if
	and only if there exists an increasing sequence $\{n_k\}$ of positive
	integers such that
	\[
	\mu(\phi^{n_k}(v)) \to 0 \quad \text{ for all } v \in T,
	\]
	and
	\[
	\mu(\phi^{-n_k}(v)) \to 0 \quad \text{ for all } v \in T^\infty.
	\]
\end{theorem}

For clarity, we prove each direction separately in what follows.

\begin{theorem}
Let $(T,\dist)$ be an unbounded, locally finite metric space with a
distinguished element $\o$, let $\mu$ be a positive function on $T$
and let $\phi$ be a self-map of $T$. Assume that $C_\phi$ is bounded
on $\Lmuz$ and $\phi$ is injective. If there exists an increasing
sequence of positive integers $\{ n_k \}$ such that 
$$
\mu(\phi^{n_k}(v)) \to 0  \qquad \text{ for every } v \in T,
$$
and
$$
\mu(\phi^{-n_k}(v)) \to 0 \qquad \text{ for every } v \in T^\infty,
$$
then $C_\phi$ is hypercyclic.
\end{theorem}

\begin{proof}
We use the Hypercyclicity Criterion. Let $X$ be the set of all finitely supported
functions on $T$. Let $f \in X$, $f$ not identically 0,
  and assume that
\[
\{ v \in T \, : \, f(v) \neq 0 \}=\{ w_1, w_2, \dots, w_s \}
\]
for some $s \in \N$.
\begin{enumerate}
\item For each $k\in\N$, we have
\begin{eqnarray*}
\| C_\phi^{n_k} f \|_\mu 
&=& \| f \circ \phi^{n_k} \|_\mu \\
&=& \sup_{v \in T} \mu(v) |f(\phi^{n_k}(v))| \\
&=& \sup_{w \in T^{n_k}} \mu(\phi^{-n_k}(w)) |f(w)| \\
&=& \max\{ \mu(\phi^{-n_k}(w_j)) |f(w_j)| \, : \,  w_j \in T^{n_k} \text{ and } j=1, 2, \dots, s \},
\end{eqnarray*}
where the maximum is understood to be 0 if 
 the set $\{ w_j \, : \, w_j \in T^{n_k}\}$ is empty. 
 Taking the limit as $
k\to \infty$ guarantees that $C_\phi^{n_k} f \to 0$, as desired.

\item Define the sequence of functions $S_n$ on $\Lmuz$ as
\[
(S_n f)(v) = \begin{cases}
f( \phi^{-n} (v)), & \text{ if } v \in \phi^n(T), \text{ and } \\
0, & \text{ if } v \notin \phi^n(T).
\end{cases}
\]
Clearly, if $f$ has finite support, then so does $S_n f$. Thus, for all $k\in\N$ 
\begin{eqnarray*}
\| S_{n_k} f \|_\mu 
&=& \sup_{v \in T} \mu(v) |(S_{n_k} f)(v))| \\
&=& \sup_{v \in \phi^{n_k}(T)} \mu(v) |f(\phi^{-n_k}(v))| \\
&=& \sup_{w \in T} \mu(\phi^{n_k}(w)) |f(w)| \\
&=& \max\{ \mu(\phi^{n_k}(w_j)) |f(w_j)| \, : \, j=1, 2, \dots, s \}.
\end{eqnarray*}
Taking the limit as $ k\to \infty$, we see 
that $S_{n_k} f \to 0$, as desired.

\item Lastly, observe that $(C_\phi^n S_n f)(v) = (S_n f) (\phi^n
  (v)) = f(v)$ for all $f \in X$ and all $v \in T$. Hence
  $C_\phi^{n_k} S_{n_k} f \to f$ as $k \to \infty$.
\end{enumerate}
By the Hypercyclicity Criterion, it follows that $C_\phi$ is hypercyclic.
\end{proof}

\begin{theorem}
Let $(T,\dist)$ be an unbounded, locally finite metric space with a
distinguished element $\o$, let $\mu$ be a positive function on $T$
and let $\phi$ be a self-map of $T$. Assume that $C_\phi$ is bounded
on $\Lmuz$ and $\phi$ is injective. If $C_\phi$ is hypercyclic, then
there exists a sequence $\{n_k\}$ of increasing positive integers such that
\[
\mu(\phi^{n_k}(v)) \to 0 \quad \text{ for all } v \in T,
\]
and
\[
\mu(\phi^{-n_k}(v)) \to 0 \quad \text{ for all } v \in T^\infty.
\]
\end{theorem}
\begin{proof}
We first make the following claim.
\begin{claim}
For every $\epsilon >0$, every $N \in \N$ and every finite set $I
\subseteq T$, there exists a natural number $n \geq N$ such that
\[
\mu(\phi^n(v)) < \epsilon \quad \text{ for every } v \in I, 
\]
and
\[
\mu(\phi^{-n}(v)) < \epsilon \quad \text{ for every } v \in I \cap T^\infty. 
\] 
\end{claim}
To prove the claim, first assume, without loss of generality, that 
\[
0<\epsilon < \min\{\mu(v) \, : \, v \in I \}.
\]
Since $C_\phi$ is hypercyclic, there exists $f \in \Lmuz$ and $n > N$
such that
\begin{equation}\label{eq_f_sum}
\bigg\| f - \sum_{v \in I} \chi_v \bigg\|_\mu < \frac{\epsilon}{2}
\end{equation}
and
\begin{equation}\label{eq_Cf_sum}
\bigg\| C_\phi^n f - \sum_{v \in I} \chi_v \bigg\|_\mu < \frac{\epsilon}{2}.
\end{equation}
Furthermore, since $\phi$ must be a run-away function (because, since
$C_\phi$ is hypercyclic, the function $\phi$ has no periodic points)
we can assume that $\phi^n(I) \cap I = \emptyset$.
	
Inequality \eqref{eq_f_sum} implies that
\[
\mu(w) |f(w)| < \frac{\epsilon}{2}, \quad \text{ for all } w \notin I.
\]
Since $\phi^n(I) \cap I = \emptyset$, we must have that $\phi^n(v)
\notin I$ for every $v \in I$ and hence 
\begin{equation}\label{eq_mu_f_phi_n}
\mu(\phi^n(v)) |f(\phi^n(v))| < \frac{\epsilon}{2}, \quad \text{ for all } v \in I.
\end{equation}
		
On the other hand, by inequality \eqref{eq_Cf_sum}, we have 
\[
\mu(v) |f(\phi^n(v))-1 | < \frac{\epsilon}{2}, \quad \text{ for all } v \in I.
\]
Therefore 
\begin{equation}\label{eq_1_f_phi_n}
1- |f(\phi^n(v))| < \frac{\epsilon}{2 \mu(v)} \quad \text{ for all } v \in I.
\end{equation}
Recalling the assumption on $\epsilon$ and combining inequalities
\eqref{eq_1_f_phi_n} and \eqref{eq_mu_f_phi_n}, we obtain
\[
0 < 1-  \frac{\epsilon}{2 \mu(v)} < |f(\phi^n(v))| <  \frac{\epsilon}{2 \mu(\phi^n(v))}
\]
for every $ v \in I$, and hence
\[
\mu(\phi^n(v)) < \frac{\epsilon}{2- \frac{\epsilon}{\mu(v)}}.
\]
Since $2- \frac{\epsilon}{\mu(v)} > 1$, it follows that
\[
\mu(\phi^n(v)) < \epsilon, \quad \text{ for all } v \in I.
\]
		
Now, inequality \eqref{eq_Cf_sum} implies that 
\ben 
\mu(w) |f(\phi^n(w))| < \frac{\epsilon}{2}, \quad \text{ for all } w \notin I.
\label{newf}\eeqn
Let $v \in I \cap T^\infty$. Observe that $\phi^{-n}(v) \notin
  I$. Indeed, if $\phi^{-n}(v)$ were in $I$, then $v$ would be in
$\phi^n(I)$ and hence $v$ would be in $\phi^n(I) \cap I \cap
T^\infty$, which is impossible since $\phi^n(I) \cap I$ is empty.  

By (\ref{newf}) applied to $\phi^{-n}(v)$, we obtain
\begin{equation}\label{eq_mu_phi_n_w}
\mu(\phi^{-n}(v)) |f(v)| < \frac{\epsilon}{2}, \quad \text{ for all } v
\in  I \cap T^\infty.
\end{equation}
		
On the other hand, by inequality \eqref{eq_f_sum}, we have
\begin{equation}\label{eq_mu_f_1_2}
\mu(v) | f(v) - 1 | < \frac{\epsilon}{2}, \quad \text{ for all } v \in I.
\end{equation}
		
Proceeding as above, from inequalities \eqref{eq_mu_phi_n_w} and
\eqref{eq_mu_f_1_2}, it follows that, for every $v \in I \cap T^\infty$, 
\[
1- \frac{\epsilon}{2 \mu(v)} <  |f(v)| <  \frac{\epsilon}{2 \mu(\phi^{-n}(v))},
\]
Therefore, for each $v\in I \cap
T^\infty$, \[\mu(\phi^{-n}(v))<\frac{\epsilon}{2-\frac{\epsilon}{\mu(v)}}<\epsilon,\]
proving the claim.

Now, for each $k \in \N$, define $I_k:=\{ v \in T \, : \, |v| \leq k
\}$. Since $T$ is locally finite, each $I_k$ is a finite set. By the
claim, there exists $n_1 \in \N$ such that
\[
\mu(\phi^{n_1}(v)) < 1, \quad \text{ for all } v \in I_1,
\]
and
\[
\mu(\phi^{-n_1}(v)) < 1, \quad \text{ for all } v \in I_1\cap T^\infty.
\]

Arguing inductively, assume that for some $s\in\N$ we have
chosen integers $n_2,\dots,n_{s-1}$ with $n_1 < n_2 < n_3 < \dots <
n_{s-1}$ such that, for $k=1, 2, \dots, s-1$, we have
\[
\mu(\phi^{n_k}(v)) < \frac{1}{k}, \quad \text{ for all } v \in I_k,
\]
and
\[
\mu(\phi^{-n_k}(v)) < \frac{1}{k}, \quad \text{ for all } v \in I_k\cap T^\infty.
\]
Now, by the claim applied to $\epsilon=\frac{1}{s}$, $N=n_{s-1}$ and 
 $I=I_s$, it follows that there exists an integer $n_s> n_{s-1}$ with
\[
\mu(\phi^{n_s}(v)) < \frac{1}{s}, \quad \text{ for all } v \in I_s,
\]
and
\[
\mu(\phi^{-n_s}(v)) < \frac{1}{s}, \quad \text{ for all } v \in I_s\cap T^\infty.
\]
	
Lastly, we observe that the sequence $\{n _k \}$ satisfies the
conclusion of the theorem.
	
Indeed, let $v \in T$ and $\epsilon >0$. Choose $N \in \N$ such
that $N \geq \frac{1}{\epsilon}$ and $N \geq |v|$. Clearly, if $k \geq
N$, then $v \in I_k$ and hence
\[
\mu(\phi^{n_k}(v)) < \frac{1}{k} \leq \epsilon.
\]
Hence, $\ds \lim_{k \to \infty} \mu(\phi^{n_k}(v))=0$, as desired.
	
Analogously, let $v \in T^\infty$ and $\epsilon >0$. Then, choose $N \in
\N$ such that $N \geq \frac{1}{\epsilon}$ and $N \geq |v|$. Clearly,
if $k \geq N$, then $v \in I_k\cap T^\infty$ and hence
\[
\mu(\phi^{-n_k}(v)) < \frac{1}{k} \leq \epsilon.
\]
Hence, $\ds \lim_{k \to \infty} \mu(\phi^{-n_k}(v))=0$, as desired.
\end{proof}


We deduce the following result for the unweighted case. Denote by
$L^0(T)$ the space $L^0_\mu(T)$ when $\mu$ is the constant function 1.

\begin{corollary}
Let $(T,\dist)$ be an unbounded, locally finite metric space with a
distinguished element $\o$, let $\phi$ be a self-map of $T$. If $\phi$
is injective and $C_\phi$ is bounded on $L^0(T)$, then $C_\phi$ is not
hypercyclic.
\end{corollary}

The following example shows that hypercyclicity can, indeed, occur. 

\begin{example}
	Let $T$ be an infinite, locally finite tree with root
        $\o$. Let $\phi$ be a bijection of $T$ with no periodic points
        (for example, if $T$ is a homogeneous tree, automorphisms with
        no periodic points clearly exist). Define a positive function
        $\mu$ on $T$ in such a way that $\mu(v) \to 0$ as $|v| \to
        \infty$ (for example, $\mu(v)=\frac{1}{|v|+1}$). Then $C_\phi$
        is hypercyclic on $\Lmuz$.
\end{example}
\begin{proof}
We claim that if $\phi$ is a bijection with no periodic points, then
$\phi^n(v) \to \infty$ as $n \to \infty$. Suppose not. Then, there
must exist a vertex $v$, a constant $M$ and an increasing sequence of
natural numbers $\{ n_k \}$ such that
$$
|\phi^{n_k}(v)| \leq M, \quad \text{ as } k \to \infty.
$$
Since the set $\{ v \in T \, : \, |v|\leq M\}$ is finite, this implies
that there must exist integers $s$ and $t$ with $n_s < n_t$ such that
$\phi^{n_s}(v)=\phi^{n_t}(v)$. But this implies that
$$
\phi^{n_t-n_s}( \phi^{n_s} (v) ) = \phi^{n_s}(v),
$$
and hence that $\phi^{n_s}(v)$ is a periodic point, which is a contradiction.
	
Now, it follows that, for the full sequence $\{ n \}$ we have that
$$
\mu(\phi^{n}(v)) \to 0, \quad \text{ as } n \to \infty,
$$
for every $v \in T$, since $\phi^{n}(v) \to \infty$.
	
Since $\phi^{-1}$ is also a bijection of $T$ without periodic points,
it follows that
$$
\mu(\phi^{-n}(v)) \to 0, \quad \text{ as } n \to \infty,
$$
for every $v \in T$. Then Theorem~\ref{t5.8} guarantees that $C_\phi$
is  hypercyclic.
\end{proof}

\section*{Acknowledgements}
The research of the third author is partially supported by the Asociaci\'on Mexicana de Cultura A.C

\bibliographystyle{amsplain}
\bibliography{references.bib}

\end{document}